\documentclass[11pt,leqno]{amsart}
\usepackage{amsmath}

\usepackage{amsfonts}
\usepackage{mathtools}
\usepackage{amssymb}
\usepackage{mathrsfs}
\usepackage{subcaption}  
\usepackage{bookmark} %(virginia:fa vedere le formule)

\usepackage[a4paper, centering]{geometry}
\usepackage{epstopdf}
\usepackage{amscd}
\usepackage[dvipsnames]{xcolor}
\usepackage{graphicx}
\usepackage{yfonts}

\newtheorem{teo}{Theorem}[section]
\newtheorem{lem}[teo]{Lemma}
\newtheorem{cor}[teo]{Corollary}
\newtheorem{prop}[teo]{Proposition}

\newcommand{\medint}{-\kern  -,395cm\int}

\theoremstyle{remark}
\newtheorem{oss}[teo]{Remark}

\theoremstyle{definition}
\newtheorem{defi}[teo]{Definition}%[section]

\theoremstyle{definition}
\newcommand{\wstar}{{\hat w}}

\newcommand{\bbR}{{\mathbb{R}}}
\newcommand{\bbN}{{\mathbb{N}}}

\usepackage{mathtools}
\usepackage{esint}

\newcommand{\R}{\bbR}

\def\de{{\rm d}}

\makeatletter
\def\cleardoublepage{\clearpage\if@twoside \ifodd\c@page\else
\hbox{}
\thispagestyle{empty}
\newpage
\if@twocolumn\hbox{}\newpage\fi\fi\fi}
\makeatother
\title{Relaxation for degenerate nonlinear functionals in the onedimensional case}

\author[V.~Chiad\`o Piat]{Valeria Chiad\`o Piat}
\address{Dipartimento di scienze Matematiche Giuseppe Lagrange, Corso duca degli abruzzi, 24, Torino (Italy)}
\email{valeria.chiadopiat@polito.it}\author[V.~De Cicco]{Virginia De Cicco}
\address{Dipartimento di Scienze di Base  e Applicate per l'Ingegneria, Sapienza Univ.\ di Roma\\
	Via A.\ Scarpa 10 -- I-00185 Roma (Italy)}
\email{virginia.decicco@uniroma1.it}
\author[A.~Melchor Hernandez]{Anderson Melchor Hernandez}
\address{Dipartimento di Matematica, Via Zamboni, 33, 40126, Bologna (Italy)}
\email{anderson.melchor@unibo.it}

\keywords{Lower semicontinuity, relaxation, degenerate variational integrals, weight, Poincar\'e inequality}
\subjclass[2020]{26A15,49J45}

\begin{document}

\begin{abstract}
In this study, we approach the analysis of a degenerate nonlinear functional in one dimension, accommodating a degenerate weight $w$. Our investigation focuses on establishing an explicit relaxation formula for a functional exhibiting $p$-growth for $1< p<+\infty$. We adopt the approach developed in \cite{CC}, where some assumptions like doubling or Muckenhoupt conditions are dropped. Our main tools consist of proving the validity of a weighted Poincar\'e inequality involving an auxiliary weight.
\end{abstract}
\maketitle

\tableofcontents
\section{Introduction}
In this work, we focus on the study of nonlinear functionals in one dimension, allowing for a degenerate weight $w$.  We aim to provide an explicit relaxation formula for a functional that exhibits $p$-growth for $1<p<+\infty$. More precisely, let us set
\begin{align*}
F(u)\coloneqq\left\{
\begin{aligned}
& \int_{\Omega} |\nabla u|^{p}w \de x\text{ \ \ if } u\in AC(\overline\Omega),\\
& +\infty \ \ \ \ \ \ \ \ \ \ \ \ \ \ \text{if}\hskip 0,1cm  u\in X\setminus AC(\overline\Omega),
\end{aligned}
\right.
\end{align*}
where $\Omega$ is an open bounded set in $\mathbb{R}$,  $w$ is a nonnegative, locally integrable function, and $X$ is a topological space comprising measurable functions which will be introduced later on. We then delve into its relaxation, aiming to provide an explicit expression of the lower semicontinuous envelope of $F$, denoted as $\overline{F}$. \\
In the last decades, some works have aimed to study the above functional by considering different functional setups; see for instance, \cite{CD,CC,FM,Ha,Ma}.  In particular, the attention had been given to the case $p=2$. This choice is canonical as it relates to probabilistic problems since in this case $F$ is interpreted as a regular Dirichlet form, guaranteeing the existence of right continuous stochastic processes, as treated in \cite{ALB,RO}. It's noteworthy that the theory of Dirichlet forms is general, with the natural ambient space being $L^{2}(\Omega,\mu)$, where $\mu\coloneqq w(x)\,\de x$. However,  the association between right continuous stochastic processes and $F$ usually requires the validity of a Poincar\'e inequality. Recent extensions have involved the analysis of weighted Sobolev spaces, incorporating geometric aspects \cite{LAMB2,LAMB3}.\\
\noindent
Since the identification of the functional ${\overline F}$ is subtle, some authors have been used the density of $C^{1}$-functions in weighted Sobolev spaces as an important tool, see for instance, \cite{CPSC,CC}.  In this approach, however, some additional assumptions on $w$, as  described in \cite{CPSC}, were necessary. For example,  to prove the density of $C^{1}$-functions, it is sometimes assumed that $w$ satisfies the doubling or Muckenhoupt conditions \cite{HKST,Muk}. Differently, in \cite{CC}, have been adopted the case where such requirements on $w$ are not satisfied, $p=2$, and where $X$ is not fixed a priori.\\
\noindent
Let us now describe our approach to relax the functional $F$ that extends to the case $1<p<+\infty$ the method used in \cite{CC}.  We underline that, a priori, we do not known the space $X$ because its precise structure depends on the choice of $w$. Specifically, for a fixed $w$, we then construct a weight $\hat{w}_{p}$ to define $X$. This function plays a crucial role in compensating for degeneracy present in $w$, and it permits to characterize the domain of $\overline{F}$ and analyze it (see Figure \ref{figura1}, for a precise example of such a function $\hat{w}_{p}$).  In our reasoning, the first step is to prove Poincar\'e inequalities involving $w$ and $\hat{w}_{p}$. Specifically, we consider the $p$-norm of the gradient term of a generic function $u$ weighted by $w$, while the $p$-norm of $u$ itself is weighted by $(\hat{w}_{p})^{p-1}$. These inequalities are referred to as Poincar\'e inequalities with double weight. Subsequently, assuming that $w$ is finitely degenerate (see Definition \ref{deff1} below), we proceed to choose $X=L^{p}((\hat{w}_{p})^{p-1})$, and we show that $AC$-functions are dense, in a suitable Sobolev space $W\subseteq X$ (see formula \eqref{eq:spaceW} below).  It is important to underline that the space $L^{p}(w)$, and $L^{p}((\hat{w}_{p})^{p-1})$ are not comparable, see Remark \ref{obsimp} below. As a consequence, we are able to determine the domain of $\overline{F}$ performing the relaxation in the strong topology of $X$. The resulting relaxed functional $\overline{F}$ maintains the same form as $F$, but its domain consists of functions that are of $L^{p}((\hat{w}_{p})^{p-1})$-integrable type. Instead, the case where $w$ is not finitely degenerate is still an open problem.\\
\noindent
This work is structured as follows. In Section \ref{sec:3}, we study the validity of Poincar\'e inequalities with double weight, see Theorem \ref{Poin1} below. In Section \ref{relaxp}, we formulate and prove our relaxation theorem, see Theorem \ref{mainfinitelocsum} below.

\section{Weighted Poincar\'e inequalities}\label{sec:3}
Let $\Omega$ be an open bounded subset of $\bbR$, and for $1<p<\infty$, we let $\frac{1}{p'}= 1-\frac{1}{p}$. We define
\begin{equation*}
F(u)=\left\{
\begin{aligned}
& \int_{\Omega} |\nabla u|^p\,w\,dx\text{   \ \ if } u\in AC(\overline\Omega)\\
& +\infty \ \ \ \ \ \ \ \ \ \ \ \text{   \ \ if } u\in X\setminus AC(\overline\Omega).
\end{aligned}
\right.
\end{equation*}
Here, $X$ is an appropriate set of integrable functions, that will be chosen in Section \ref{sec:3}. Further, let $\overline F:\,X\to [0,+\infty]$ denote the lower semicontinuous envelope of $F$ w.r.t. the topology of  $X$. We consider a weight $w:\R \to \R$ satisfying
\begin{equation}\label{minassweight}
w \ge 0 \text{ a.e.} ,\, \ w \in L^1_{\rm{loc}}(\bbR).
\end{equation}
From now on, it is not restrictive to assume that $\Omega=(a,b)$ is a bounded open interval, and we consistently interchange $\Omega$ and $(a,b)$ throughout the text. We denote by $I_{p,\Omega, w}$ the set
\begin{align*}
I_{p,\Omega, w}:=\Big\{x\in\Omega:\,&\exists\,\epsilon >0\text{ such that }
w^{-\frac{1}{p-1}}\in L^1\left((x-\epsilon ,x+\epsilon )\right)\Big\}.
\end{align*}
The set $I_{p,\Omega, w}$ is the largest open set in $\Omega$ such that $w^{-\frac{p'}{p}}$ is locally summable. It is noteworthy that this set has already been considered in \cite{CC} for $p=2$. In this work, we are exploring a $p$-version of the results studied in that work. Without loss of generality, we 
%assume the existence of two countable sets ${a_{p,i}}$ and ${b_{p,i}}$ such that $a \leq a_{p,i} < b_{p,i} \leq b$. Consequently, the intervals $(a_{p,i}, b_{p,i})$ are disjoint, and we 
can express $I_{p,\Omega, w}$ as the disjoint union
\begin{equation}\label{minassweight3}
I_{p,\Omega, w} = \bigcup_{i=1}^{{N_{p,w}}} (a_{p,i}, b_{p,i}),
\end{equation}
with $1\leq N_{p,w}\leq +\infty$. Subsequently, for the sake of a lean notation, we set $a_{i} \coloneqq a_{p,i}$, $b_{i} \coloneqq b_{p,i}$, $N_{w} \coloneqq N_{p,w}$, and $I_{\Omega,w} \coloneqq I_{p,\Omega,w}$. Our objective is to characterize the relaxation of the functional  $F$ concerning $L^{p}(\Omega,(\hat{w}_{p})^{p-1})$-convergence, where $\hat w$ is a suitable weight (see \eqref{pesop}).  To achieve this, we reintroduce the concept of a finitely degenerate weight, as discussed in \cite{CC}.

\begin{defi}\label{deff1}
\begin{itemize}
\item[(i)] If $I_{\Omega,w}=\,\emptyset$, we put $N_w:=0$.
\item[(ii)]  If $1\le\,N_w<\,+\infty$  we say that $w$ is {\it finitely degenerate} in $\Omega$. 
\item[(iii)]   If $N_w=\,\infty$  we say that $w$ is {\it not finitely degenerate} in $\Omega$.
\end{itemize}
\end{defi}
Furthermore, we define the set
\begin{equation}\label{domw}
{\rm{Dom}}_w:=\Big\{u:\Omega\to\R: u\in W^{1,1}_{\text{loc}}(I_{\Omega, w}),
\int_{I_{\Omega, w}} |u' |^p\,w\,dx<+\infty\Big\}\,.
\end{equation}
As we will see below, this set must be the core of the relaxed functional of $F$ with respect to $L^{p}(\Omega,(\hat w_{p})^{p-1})$-convergence.

\begin{oss}
We note that, if $w^{-\tfrac{1}{p-1}}\in L^1((a, b)),$ then, obviously, $w$ is finitely degenerate in
$\Omega$ with $N_w = 1.$ In this case
$${\rm{Dom}}_w = \{u \in AC([a, b]) :
\int_a^b
|u'|^p w\, dx < +\infty\}$$
(see \cite{CPSC})  since $AC([a, b]) = W^{1,1}([a, b])$.
\end{oss}

\begin{lem}[Fundamental convergence]\label {lemma2}
Let $(u_h)_h\subset AC([a,b])$ such that
\begin{enumerate}
\item[(a)] $\displaystyle\sup_{h\in\bbN}\int_{I_{\Omega,w}}|u_h'|^{p}w\,\de x<+\infty$\,,
\item[(b)] for every $i=\,1,\dots,N_w$ there exists $c_i$ such that $a_i< c_i<b_i$ and there exist finite the following limits
$$
\lim_{h\to+\infty}u_h(c_i)=d_i\in \bbR\,.
$$
\end{enumerate}
Then there exists a subsequence $(u_{h_k})$ and a function 
$u:\,I_{\Omega,w}\to\bbR$ such that 
\begin{enumerate}
\item[(i)] $\displaystyle\lim_{k\to+\infty}u_{h_k}(x) =u(x) $ for every $x\in I_{\Omega,w}$\,,
\item[(ii)] $u\in {\rm{Dom}}_w$,
\item[(iii)] $\displaystyle\int_{I_{\Omega,w}} |u' |^{p}\,w\,\de x\leq \liminf_{k\to+\infty}\int_{I_{\Omega,w}} |u_{h_k}' |^{p}\,w\,\de x$\,.
\end{enumerate}
\end{lem}
\proof
Let us notice that the proof of this Lemma coincides with the one given in \cite[Lemma 4.3]{CC}. We only need to notice that $L_{\rm loc}^{p}(I_{\Omega,w})\subset L_{\rm loc}^{1}(I_{\Omega,w})$, and the conclusion of our Lemma follows.
\qed

\bigskip

\subsection{An auxiliary weight}
In what follows, we construct a suitable weight $\hat{w}_{p}$ for $1<p<+\infty$ for which it is possible to prove a Poincar\'e inequality involving $w$ and $(\hat{w}_{p})^{p-1}$.
Let $w:\,\Omega=(a,b)\to [0,\infty)$ be a function satisfying \eqref{minassweight} and \eqref{minassweight3}. We let $\wstar_{p}:\Omega\to [0,+\infty[$ be defined as
\begin{equation}\label{pesop}
\wstar_{p}(x):=
\begin{cases}
\displaystyle\lim_{x\to a_i^+}  \left(\int_{x}^{{a_i+b_i}\over{2}}\frac{1}{w^{\frac{1}{p-1}}(y)}\,dy\right)^{-1}    &\text{ if }x=a_i \\
\left(\int_{x}^{{a_i+b_i}\over{2}}\frac{1}{w^{\frac{1}{p-1}}(y)}\,dy\right)^{-1} &\text{ if } {a_i}< x\leq {{3a_i+b_i}\over{4}}\\
\left(\int_{{3a_i+b_i}\over{4}}^{{a_i+3b_i}\over{4}}\frac{1}{w^{\frac{1}{p-1}}(y)}\,dy\right)^{-1} &\text{ if } {{3a_i+b_i}\over{4}}\leq x\leq {{a_i+3b_i}\over{4}}\\
\left(\int_{{a_i+b_i}\over{2}}^x\frac{1}{w^{\frac{1}{p-1}}(y)}\,dy\right)^{-1} &\text{ if } {{a_i+3b_i}\over{4}}\leq x<b_i\\
\displaystyle\lim_{x\to b_i^-}  \left(\int_{{a_i+b_i}\over{2}}^{x}\frac{1}{w^{\frac{1}{p-1
}}(y)}\,dy\right)^{-1}    &\text{ if } x=b_i\\
\ \ \ \ \ \ \ \ \ 0&\text{ if } x\in \Omega\setminus \overline{I_{\Omega,w}}\,.
\end{cases}
\end{equation}
\begin{oss}\label{newoss}
Let us notice that $\hat{w}_{p}$ heavily depends on $p$. In this part, $\hat{w}_{p}$ is defined as the inverse of an integral term, which allows us to use its nice properties, such as continuity, that we used to prove the Proposition 2.4. 
\end{oss}
In the next, we give an explicit example of the previous function $\hat{w}_{p}$. 
\begin{figure}[!ht]
\centering
\begin{subfigure}[b]{0.4\textwidth}
\includegraphics[width=\textwidth]{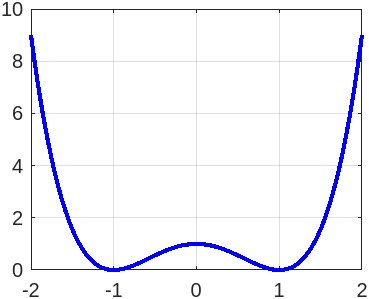}
\end{subfigure}
\hfill
  \begin{subfigure}[b]{0.37\textwidth}
    \includegraphics[width=\textwidth]{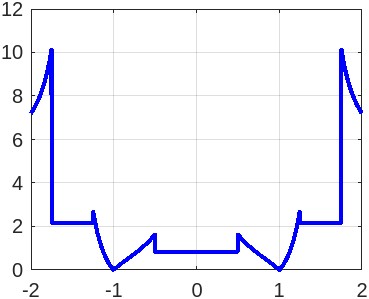}
  \end{subfigure}
  \caption{In the first figure on the left hand side, we have the profile of $w(x)=(1-x^{2})^{2}$ for  $x\in (-2,2)$, while in the right hand side, we have its associated weight $\hat{w}_{2}$. In this case, we note that $N_{w}=3$.}
  \label{figura1}
\end{figure}
In the relaxation of $F$, we will consider $L^{p}(\Omega, (\hat{w}_{p})^{p-1})$-convergence. Before providing the precise details of how we relax $F$, let us gather some properties of the functions $\wstar_{p}$ in the following proposition. The proof is elementary, taking the definitions into account.
\begin{prop}\label{propw}
\begin{itemize}\
\item[(i)] Suppose that $w^{-\frac{1}{p-1}}$ is  not locally summable in $\Omega$, that is, $I_{\Omega,w}=\,\emptyset$. Then $\wstar_{p}\equiv 0$.
\item[(ii)] For each $i=1,\dots, N_w$, $\wstar_{p}$ is constant in $[{{3a_i+b_i}\over{4}},{{a_i+3b_i}\over{4}}]$, increasing in $[{a_i},{{3a_i+b_i}\over{4}}]$, decreasing in $[{{a_i+3b_i}\over{4}},b_i]$ and absolutely continuous in each interval.  In particular, the following hold true:
\[
0\,<\wstar_{p}(x)\le\,\sup_{y\in (a_i,b_i)}\wstar_{p}(y)<\,\infty\quad\forall\,x\in (a_i,b_i)\,,
\]
\[
\inf_{x\in [\alpha,\beta]}w(x)>\,0\text{ for each $x\in [\alpha,\beta]$, $a_i<\,\alpha<\,\beta<\,b_i$, }
\]
and $\wstar_{p}(a_i)=\,0$ (respectively $\wstar_{p}(b_i)=\,0$) if and only if $w^{-\frac{1}{p-1}}\notin L^1((a_i,\frac{a_i+b_i}{2}))$ (respectively $w^{-\frac{1}{p-1}}\notin L^1((\frac{a_i+b_i}{2},b_i))$). 
\item[(iii)] We have
\begin{equation}\label{derwstar}
(\wstar_{p})^\prime
=
{{\left(\wstar_{p}\right)^2}\over{w^{\frac{1}{p-1}}}}\quad\text{ a.e. in }\left({a_i},{{3a_i+b_i}\over{4}}\right)\cup \left({{a_i+3b_i}\over{4}},b_i\right)\,.
\end{equation}
\item[(iv)] Suppose that $w^{-\frac{1}{p-1}}\in L^1(\Omega)$. Then there exists a  constant $c>\,0$ such that
\[
0<\,\frac{1}{c}\le\,\wstar_{p}(x)\le\,c\quad\text{ a.e. } x\in\Omega\,.
\]
\item[(v)] Suppose that $w$ is finitely degenerate in $\Omega$, that is, \eqref{minassweight3} holds with $1\le\,N_w<\,\infty$. Then there exists a  constant $c>\,0$ such that
\[
0\le\,\wstar_{p}(x)\le\,c\quad\text{ a.e. }x\in\Omega\,.
\]
\item[(vi)] Suppose that $w$ is not finitely degenerate in $\Omega$, that is, \eqref{minassweight3} holds with $N_w=\,\infty$. Then $\wstar_{p}\in L^\infty_{\rm loc}(I_{\Omega,w})$.
\end{itemize}
\end{prop}

\begin{oss}
Let us consider the following example already considered in \cite[Example 4.7]{CC}. Let us take $\alpha>0$ such that $\frac{\alpha}{p-1}>1$. Here, we  also observe that if $w$ is not finitely degenerate in some open set $\Omega$, then it can also happens that $\wstar_{p}\notin L^{1}(\Omega)$. Suppose that $\Omega=(0,1)$ and let $(a_{i},b_{i})$, $i=1,\ldots, +\infty,$ be a sequence of disjoint open intervals in $(0,1)$ and $m_{i}$ be a sequence of positive real numbers which will be fixed later on. Let $w:(0,1)\rightarrow [0,+\infty)$ defined as follows:
\begin{align*}
w(x)\coloneqq
\begin{cases}
m_{i}(x-a_{i})^{\alpha}& \hskip 0,1cm \text{if $a_{i}\leq x\leq \frac{a_{i}+b_{i}}{2}$,}\\
m_{i}(b_{i}-x)^{\alpha}& \hskip 0,1cm \text{if $\frac{a_{i}+b_{i}}{2}\leq x\leq b_{i}$,}\\
0&\hskip 0,1cm \text{otherwise}.
\end{cases}
\end{align*}
Let us fix $a_{i}\leq x\leq \frac{3a_{i}+b_{i}}{4}$. Then by definition of $\wstar_{p}$ we have
\begin{align*}
\wstar_{p}(x)=\frac{(\alpha_{p}-1)m_{i}^{\frac{1}{p-1}}(x-a_{i})^{\alpha_{p}-1}}{1-\left(\frac{2(x-a_{i})}{b_{i}-a_{i}}\right)^{\alpha_{p}-1}},
\end{align*}
where $\alpha_{p}\coloneqq \frac{\alpha}{p-1}$. Let us notice that $0\leq \frac{2(x-a_{i})}{b_{i}-a_{i}}\leq \frac{1}{2}$, and then 
\begin{align*}
(\alpha_{p}-1)m_{i}^{\frac{1}{p-1}}(x-a_{i})^{\alpha_{p}-1}\leq \wstar_{p}(x)\leq \frac{(\alpha_{p}-1)m_{i}^{\frac{1}{p-1}}(x-a_{i})^{\alpha_{p}-1}}{1-\left(\frac{1}{2}\right)^{\alpha_{p}-1}}.
\end{align*}
Therefore,
\begin{align*}
\wstar_{p}(x)\cong m_{i}^{\frac{1}{p-1}}(x-a_{i})^{\alpha_{p}-1},\ \ a_{i}\leq x\leq \frac{3a_{i}+b_{i}}{4},
\end{align*}
and
\begin{align*}
\int_{0}^{\frac{3a_{i}+b_{i}}{4}}\wstar_{p}(x)\de x\cong m_{i}^{\frac{1}{p-1}}(b_{i}-a_{i})^{\alpha_{p}}.
\end{align*}
Hence, if we take the set $\{m_{i}: i\in\bbN\}$ such that 
\begin{align}\label{divseries}	
\sum_{i=1}^{+\infty}m_{i}^{\frac{1}{p-1}}(b_{i}-a_{i})^{\alpha_{p}}=+\infty
\end{align}
we get that $\wstar_{p}\notin L^{1}((0,1))$. An example of function $w$ for which the latter series diverges can be given in the following manner. Suppose that for each $i\in\bbN$, $b_{i}-a_{i}=\frac{1}{2^{i}}$, and we choose $m_{i}\coloneqq 2^{(i+1)\alpha}$. Since $\alpha>0$, we get that $m_{i}^{\frac{1}{p-1}}>2^{i\alpha_{p}}$, and thus \eqref{divseries} holds true.

\end{oss}

\begin{oss}\label{obsimp}
We recall another example already considered in \cite[Example 5.1]{CC} in order to see that the Lebesgue spaces $L^p(\Omega,w)$ and $L^p(\Omega,(\hat{w}_{p})^{p-1})$ are independent. 
Let $\Omega=(0,1)$, $w:(0,1)\to (0,\infty)$, $w(x)=x^\alpha$, with $-1<\alpha< p-1$, $w\in L^1((0,1))$ and 
\begin{equation*}\label{strucwstar4}
\frac1{w^{\frac1{p-1}}}\in L^1((0,1))
\end{equation*}
and so, by (iv) in Proposition \ref{propw}, it follows that
\begin{equation*}\label{strucwstar5}
L^p(\Omega,(\hat{w}_{p})^{p-1})=\,L^p(\Omega)\,.
\end{equation*}
Therefore, if $\alpha\in (0,p-1)$, since $w(x)<\,1$ for each $x\in (0,1)$,
\[
L^p(\Omega,(\hat{w}_{p})^{p-1})=\,L^p(\Omega)\subsetneq L^p(\Omega,w)\,;
\]
if $\alpha=\,0$, since $w(x)=\,1$ for each $x\in (0,1)$,
\[
L^p(\Omega,(\hat{w}_{p})^{p-1})=\,L^p(\Omega)= L^p(\Omega,w).
\]
Let $1<p<2$; if $\alpha\in (-p+1,0)$, since $w(x)>\,1$ for each $x\in (0,1)$,
\[
L^p(\Omega,\wstar_p)=\,L^p(\Omega)\supsetneq L^p(\Omega,w)\,.
\]
\end{oss}

\subsection{A Poincar\'e inequality with a double weight}
In what follows, we derive a weighted Poincar\'e inequality that we use later on. We first state some preliminary lemmas.

\begin{prop}\label {poincarep}
Let us consider a fixed $u\in {\rm{Dom}}_w$, $i=1,\dots,N_w$, and let $\frac{1}{p}+\frac{1}{p\prime}=1$. Let us take $\eta,x$ such that $a_i<\eta\leq x\leq {{a_i+b_i}\over{2}}$. The following hold true:

\begin{equation}\label{a1}
|u(x)-u(\eta)|\,\sqrt[p\prime]{\wstar_{p}(\eta)}
\leq \left(\int_{\eta}^x|u'(y)|^{p}w(y)\,\de y\right)^{1\over p}\,;
\end{equation}

\begin{equation}\label{a2}
|u(\eta)|^{p}(\wstar_{p}(\eta))^{p-1}\leq 2^{p-1}\left(|u(x)|^{p}(\wstar_{p}(\eta))^{p-1}+\int_{a_i}^{x}|u'(y)|^{p}w(y)\,\de y\,\right).
\end{equation}
Let us take $\eta,x$ such that ${{a_i+b_i}\over{2}}\leq x\leq \eta<b_i$. The following hold true:
\begin{equation}\label{a3}
|u(x)-u(\eta)|\,\sqrt[p\prime]{\wstar_{p}(\eta)}
\leq \left(\int_{x}^\eta|u'(y)|^{p}w(y)\,\de y\right)^{1\over {p}}\,;
\end{equation}

\begin{equation}\label{a4}
|u(\eta)|^{p}(\wstar_{p}(\eta))^{p-1}\leq 2^{p-1}\left(|u(x)|^{p}(\wstar_{p}(\eta))^{p-1}+\int_{x}^{b_i}|u'(y)|^{p}w(y)\,\de y\,\right).
\end{equation}

\end{prop}

\proof. 
By definition of ${\rm{Dom}}_w$, and by the immersion of $W^{1,1}(I_{\Omega,w})$ into $AC(I_{\Omega,w})$ , we also have that $u\in AC_{\rm loc}((a_i,b_i))$. Then for every $x, \eta\in]a_i,{{a_i+b_i}\over{2}}]$ such that $a_i<\eta\leq x\leq {{a_i+b_i}\over{2}}$ we have
\begin{equation}\label{a5}
\begin{split}
|u(x)-u(\eta)|=&\left|\int_\eta^{x}u'(y)\,\de y\right|\leq
\left(\int_{\eta}^x|u'(y)|^{p}w(y)\,\de y\right)^{1\over p}
\left(\int_{\eta}^{x} w^{-\frac{p'}{p}}(y)\,\de y\right)^{1\over p'}\\
&  
\leq \left(\int_{\eta}^x|u'(y)|^{p} w(y)\,\de y\right)^{1\over p}
\left(\int_{\eta}^{{{a_i+b_i}\over{2}}}w^{-\frac{p'}{p}}(y)\,\de y\right)^{1\over p'}\,.
\end{split}
\end{equation}
Let us noticing that, if $a_i<\eta\le\min\{\frac{3a_i+b_i}4,x\}$, then \eqref{a1} follows by \eqref{a5} and the definition of $\wstar$. Furthermore, if $\frac{3a_i+b_i}4\le \eta\le\,\frac{a_i+b_i}2$, since we have that
\[
\left(\int_{\eta}^{{{a_i+b_i}\over{2}}}\frac{1}{w^{\frac{p'}{p}}(y)}\,\de y\right)^{1\over p\prime}\le\,\frac{1}{\sqrt[p\prime]{\wstar_{p}(\eta)}}\,,
\]
\eqref{a1} still follows by \eqref{a5} and the definition of $\wstar_{p}$.
Then, since
\begin{equation*}
|u(\eta)|^{p}\leq2^{p-1}\left(|u(x)|^{p}+|u(\eta)-u(x)|^{p}\right)\,,
\end{equation*}
by \eqref{a1}, we then deduce \eqref{a2}. The remaining formulas \eqref{a3} and (\ref{a4}) follow  by arguing in a similar way.
\qed

\bigskip
Some consequences of Proposition \ref{poincarep} are summarized in the following Corollary.
\begin{cor}\label{proppp1}
Let us fix $u\in {\rm{Dom}}_w$, and $i=1,\dots,N_w$. Then the following hold true:
\begin{itemize}
\item[(i)]
$\displaystyle
|u(\eta)|^{p}(\wstar_{p}(\eta))^{p-1}
\leq 2^{p-1}\left(\left|u\left(\frac{a_i+b_i}{2}\right)\right|^{p}\,(\wstar_{p}(\eta))^{p-1}
+\int_{a_i}^{b_i}|u'(y)|^{p}w(y)\,\de y\right)\,,
$
for each $\eta\in (a_i,b_i)$. Furthermore, $u\in L^{p}((a_i,b_i),(\wstar_{p})^{p-1})$, and if $N_{w}<+\infty$  then $u\in L^{p}(\Omega,(\wstar_{p})^{p-1})$.
\item[(ii)] Suppose that $\displaystyle{\int_{a_i}^{\frac{a_i+b_i}{2}}\frac{1}{w^{\frac{1}{p-1}}}\de x=+\infty}$ (respectively, suppose that $\displaystyle{\int_{\frac{a_i+b_i}{2}}^{b_i}\frac{1}{w^{\frac{1}{p-1}}}\de x=+\infty}$). Then there exists 

$$\displaystyle\lim_{x\to a_i^+} (u^{p} \,(\wstar_{p})^{p-1})(x)=\,0 \,(\text{respectively,}\lim_{x\to b_i^-} (u^{p} \,(\wstar_{p})^{p-1})(x)=\,0)\,.
$$
\item[(iii)] Suppose that $\displaystyle{\int_{a_i}^{\frac{a_i+b_i}{2}}\frac{1}{w^{\frac{1}{p-1}}}\de x<\,\infty}$ (respectively, suppose that $\displaystyle{\int_{\frac{a_i+b_i}{2}}^{b_i}\frac{1}{w^{\frac{1}{p-1}}}\de x<\,\infty}$). Then 
\[u\in AC\Big(\big[a_i,\frac{a_i+b_i}{2}\big]\Big) \,(\text{respectively, }u\in AC\Big(\big[\frac{a_i+b_i}{2},b_i\big]\Big)\,.
\]
\end{itemize}
\end{cor}

\proof
(i) Note that by (\ref{a1}) and  (\ref{a2}) with $x={{a_i+b_i}\over{2}}$, we can obtain the desired inequality. We now aim to justify (ii). Let $  a_i<\eta\leq x\leq \frac{a_i+b_i}{2}$.  Suppose that $\displaystyle{\int_{a_i}^{\frac{a_i+b_i}{2}}\frac{1}{w^{\frac{1}{p-1}}}\de x=+\infty}$. By the definition of $\wstar_{p}$, we obtain that $\lim_{\eta\to a_i^+}\wstar_{p}(\eta)=0$.
Furthermore,  for each $x\in (a_i,\frac{a_i+b_i}{2})$, we have that by \eqref{a2} the following inequality holds true:
$$
\limsup_{\eta\to a_i^+}|u(\eta)|^{p} (\wstar_{p}(\eta))^{p-1}\leq 2^{p-1}\int_{a_i}^x|u'(y)|^{p}w\,\de y.
$$
Hence, by letting $\lim$ as $x\to a_i^+$ in the previous inequality, then
that
\[
\lim_{\eta\to a_i^+}|u(\eta)|^{p}(\wstar_{p}(\eta))^{p-1}= 0\,.
\]
Now, let us suppose that $\displaystyle{\int_{\frac{a_i+b_i}{2}}^{b_i}\frac{1}{w^{\frac{1}{p-1}}}\de x=+\infty}$. Then, we immediately obtain that
\[
\lim_{\eta\to b_i^-}|u(\eta)|^{p}(\wstar_{p}(\eta))^{p-1}= 0\,.
\]
(iii)  Lastly, let us suppose that $\displaystyle{\int_{a_i}^{\frac{a_i+b_i}{2}}\frac{1}{w^{\frac{1}{p-1}}}\de x<\,\infty}$. We now prove that $u\in AC\big(\big[a_i,\frac{a_i+b_i}{2}\big]\big)$. Since $u\in AC([a_i+\delta,\frac{a_i+b_i}{2}])$, for each $\delta>0$, it is sufficient to prove that there exists the following limit
\begin{equation}\label{popol}
\lim_{\eta\to a_i^+}u(\eta)\in\R.
\end{equation}
Let us first show that
\begin{equation}\label{rtrt}
u'\in L^1\left(\left(a_i,\frac{a_i+b_i}2\right)\right).
\end{equation} 
Indeed, let us notice that by H\"{o}lder inequality, it holds that
\begin{align*}
\int_{a_{i}}^{\frac{a_{i}+b_{i}}{2}}\vert u'(x)\vert \de x\leq \left(\int_{a_{i}}^{\frac{a_{i}+b_{i}}{2}}\vert u'(x)\vert^{p}w(x)\de x \right)^{\frac{1}{p}}\left(\int_{a_{i}}^{\frac{a_{i}+b_{i}}{2}}w(x)^{-\frac{1}{p-1}}\de x\right)^{\frac{1}{p'}}<+\infty.
\end{align*}
On the other hand, by the fundamental theorem of Calculus for every $\eta\in(a_i,\frac{a_i+b_i}{2}]$
\begin{equation}\label{mbmb}
u(\eta)=u\Big(\frac{a_i+b_i}{2}\Big)-\int_\eta^{\frac{a_i+b_i}{2}}u'(x) \de x.
\end{equation}
Therefore, by \eqref{rtrt} and \eqref{mbmb},we then deduce the existence of the desired limit \eqref{popol}. Let us observe that the remaining case is similar.
\qed

\bigskip
We now aim to prove the validity of a Poincar\'e type inequality with respect to the weight function $(\hat{w}_{p})^{p-1}$.
\begin{teo}[Poincar\'e type inequality on ${\rm{Dom}}_w$]\label{Poin1}
Let $1\leq N_{w}\leq +\infty$. For every $u\in {\rm{Dom}}_w$,  it holds true that
\begin{equation}\label{poincareformula}
\sum_{i=1}^{N_{w}}\medint_{a_i}^{b_i}
\left|u(\eta)-u\left(\frac{a_i+b_i}{2}\right)\right|^{p}(\wstar_{p}(\eta))^{p-1}\,\de \eta\leq \int_{I_{\Omega,w}}|u'(y)|^{p}w(y)\,\de y\,.
\end{equation}
\end{teo}
\proof
Let us set $1\leq i \leq N_{w}$, and consider $x=\frac{a_i+b_i}{2}$ in \eqref{a1}. Then one gets
\begin{equation*}
\left|u(\eta)-u\left(\frac{a_i+b_i}{2}\right)\right|^{p}(\wstar_{p}(\eta))^{p-1}\leq
\int_{a_i}^{\frac{a_i+b_i}{2}}|u'(y)|^{p}w(y)\,\de y.
\end{equation*}
 Hence, by integrating with respect to $\eta$ gives that
\begin{equation*}
\int_{a_i}^{\frac{a_i+b_i}{2}}
\left|u(\eta)-u\left(\frac{a_i+b_i}{2}\right)\right|^{p}(\wstar_{p}(\eta))^{p-1}\,\de \eta\leq\frac{b_i-a_i}{2}
\int_{a_i}^{\frac{a_i+b_i}{2}}|u'(y)|^{p}w(y)\,\de y.
\end{equation*}
Further, by letting the same reasoning, we get
\begin{equation*}
\int_{\frac{a_i+b_i}{2}}^{b_i}
\left|u(\eta)-u\left(\frac{a_i+b_i}{2}\right)\right|^{p}(\wstar_{p}(\eta))^{p-1}\,\de \eta\leq\frac{b_i-a_i}{2}
\int_{\frac{a_i+b_i}{2}}^{b_i}|u'(y)|^{p}w(y)\,\de y.
\end{equation*}
Therefore, by combining both inequalities, we deduce that
\begin{equation*}
\int_{a_i}^{b_i}
\left|u(\eta)-u\left(\frac{a_i+b_i}{2}\right)\right|^{p}(\wstar_{p}(\eta))^{p-1}\,\de\eta\leq(b_i-a_i)
\int_{a_i}^{b_i}|u'(y)|^{p}w(y)\,\de y,
\end{equation*}
and thus
\begin{equation*}
\medint_{a_i}^{b_i}
\left|u(\eta)-u\left(\frac{a_i+b_i}{2}\right)\right|^{p}(\wstar_{p}(\eta))^{p-1}\,\de \eta\leq
\int_{a_i}^{b_i}|u'(y)|^{p}w(y)\,\de y.
\end{equation*}
Since $u\in {\rm{Dom}}_w$, then
\begin{equation*}
\sum_{i=1}^{N_{w}}
\int_{a_i}^{b_i}|u'(y)|^{p}w(y)\,\de y=\int_{I_{\Omega,w}}|u'(y)|^{p}w(y)\,\de y<+\infty,
\end{equation*}
and our conclusion follows.
\qed

In what follows, we consider the following set

\begin{align}\label{eq:spaceW}
W=W(\Omega, w):=\,{\rm{Dom}}_w\cap L^{p}(\Omega,(\wstar_{p})^{p-1}).
\end{align}
In the next proposition
we prove that $W$ endowed with a suitable norm is a Banach space and some related properties. 

\begin{prop} 
Let us consider $W$ be defined as in \eqref{eq:spaceW}, and endow it with the norm
\[
\|u\|_W\coloneqq \sqrt[p]{\|u\|_{L^{p}(I_{\Omega,w},(\wstar_{p})^{p-1})}^{p}+\|u'\|_{L^{p}(I_{\Omega,w},w)}^{p}}\quad\text{ if }u\in W\,.
\]
Then $(W,\|u\|_W)$ is a Banach space. Further, if $w$ is a finitely degenerate weight, then
\begin{equation}\label{LipdenseW}
\text{ $AC(\overline\Omega)$ is dense in $(W,\|\cdot\|_W)$} 
\end{equation}
in the following sense. For every $u\in W$ there exists a sequence $(u_h)_h\subset AC(\overline\Omega)$ such that
\begin{equation}\label{approxseqW}
\lim_{h\to\infty}u_h=\,u \text{ in } (W,\|\cdot\|_W)\,,
\end{equation}
that is,
\begin{equation}\label{convinWsplitted}
\lim_{h\to\infty}u_h=\,u \text{ in } L^{p}(\Omega,(\wstar_{p})^{p-1}),\text{ and }\lim_{h\to+\infty}u'_h=\,u' \text{ in } L^{p}(I_{\Omega,w},w)\,.
\end{equation}
\end{prop}

\begin{proof} Let us first prove that $W$ is a Banach space.
Suppose that $(u_h)_h\subset (W, \|\cdot\|_W)$  is a Cauchy sequence. Hence, by definition, we have that $(u_h)_h\subset L^{p}(I_{\Omega,w},(\wstar_{p})^{p-1})$, and $(u'_h)_h\subset L^{p}(I_{\Omega,w}, w)$ are Cauchy sequences.  Since $L^{p}(I_{\Omega,w},(\wstar_{p})^{p-1})$, and $L^{p}(I_{\Omega,w},w)$ are complete spaces, it follows that there exist $u\in L^{p}(I_{\Omega,w},(\wstar_{p})^{p-1})$, and $v\in L^{p}(I_{\Omega,w},w)$ such that,
\begin{equation}\label{CsW}
u_h\to u\text{ in  }L^{p}(I_{\Omega,w},(\wstar_{p})^{p-1}),\text{ and }u'_h\to v\text{ in  }L^{p}(I_{\Omega,w},w)\,,
\end{equation}
as $h\to+\infty$. In what follows, we need to prove that for each $i=1,\dots,N_w$,
\begin{equation}\label{uinD*}
u\in AC_{\rm loc}((a_i,b_i)) \text{ and } u'=\,v\text{ a.e.  in } (a_i,b_i)\,,
\end{equation}
and from which we conclude that $u\in {\rm{Dom}}_w$, and thus proving that $W$ is a Banach space.  Let us fix $i=1,\dots,N_w$, and let $a_i<\,\alpha<\,\beta<\,b_i$.  Since $w^{-\frac{1}{p-1}}\in L_{\rm loc}^{1}(I_{\Omega,w})$, together to Proposition \ref{propw} (ii), it follows that

\begin{equation}\label{w1inL1}
\frac{1}{\wstar_{p}}\in L^1( (\alpha,\beta)).
\end{equation}
Then by \eqref{w1inL1}, we get the continuous inclusions
\begin{equation}\label{inclusion}
L^{p}((\alpha,\beta),(\wstar_{p})^{p-1})\subset L^1((\alpha,\beta)),\text{ and } L^{p}((\alpha,\beta), w)\subset L^1((\alpha,\beta))\,.
\end{equation}
Indeed, if $f\in L^{p}((\alpha,\beta),(\wstar_{p})^{p-1})$, then

\begin{align*}
\int_{\alpha}^{\beta}\vert f(x)\vert \de x\leq \left(\int_{\alpha}^{\beta}\vert f(x) \vert^{p}(\wstar_{p}(x))^{p-1} \de x\right)^{\frac{1}{p}}\left( \int_{\alpha}^{\beta}(\wstar_{p}(x))^{-1}
\de x\right)^{\frac{1}{p\prime}}<+\infty,
\end{align*}
which proves \eqref{inclusion}. Therefore, by \eqref{CsW}
\begin{equation*}
u_h\to u\text{ in  }L^1((\alpha,\beta)),\text{ and }u'_h\to v\text{ in  }L^1((\alpha,\beta))\,,
\end{equation*}
and thus 
\[
u\in AC([\alpha,\beta]), \text{ and } u'=\,v\text{ a.e.  in } [\alpha,\beta]\,,
\]
from which  we conclude that \eqref{uinD*} holds true.

\bigskip
 We now need to show that \eqref{LipdenseW} is true. To this end, it is enough to show that for each $u\in W$, there exists a sequence $(\bar u_h)_h\subset AC(\overline\Omega)$ such that 
\begin{equation}\label{recovsequbd}
\bar u_h\rightarrow u\text{ in }L^{p}(\Omega,(\wstar_{p})^{p-1})\text{ and } (\bar u'_h)_h\text{ bounded in } L^{p}(I_{\Omega,w},w)\,.
\end{equation}
In fact, by \eqref{recovsequbd} and the fact that $W$ is a reflexive space, up to a subsequence, we can assume that
\begin{equation*}
\bar u_h\rightarrow u\text{ in $W$-weak}\,.
\end{equation*}
Thanks to Mazur's lemma, there exists a sequence $(u_h)_h$ such that, for each $h$, there exist real numbers $c_{h,j}\in\R$  and $h_j\in\bbN$ for $j=1,\dots, m_h$, with
 \begin{equation*}
 \sum_{j=1}^{m_h} c_{h,j}=1\text{ and } u_h=\,\sum_{j=1}^{m_h} c_{h,j}u_{h_j}\in AC(\overline\Omega)\,,
 \end{equation*}
such that \eqref{approxseqW} holds true. Thanks to \eqref{LipdenseW}, we only need to show that \eqref{recovsequbd} is true. 

\bigskip
Let us consider $u\in W$, so that $u'\in L^{p}(I_{\Omega,w},w)$. It is well-known that there exists a sequence of functions $(v_h)_h\subset C^0_c(I_{\Omega,w})\subset L^{p}(I_{\Omega,w},w)$ such that
\begin{equation}\label{nuova1}
\|v_h-u'\|^{p}_{L^{p}(I_{\Omega,w},w)}=\,\sum_{i=1}^{N_w} \int_{a_i}^{b_i}|v_h-u'|^{p}\,w\,\de x\to 0,\text{ as }h\to +\infty\,.
\end{equation}
In what follows, we define for each $h\in\bbN$, the function $\tilde{u}^{(i)}_h:(a_i,b_i)\to\R$, $i=1,2,\dots,h$ as
\begin{equation}\label{nuova2}
\tilde{u}^{(i)}_h(x) :=
u\left(\frac{a_i+b_i}{2}\right)-\int_x^{\frac{a_i+b_i}{2}} v_h(y)\,\de y\,,\ \quad x\in (a_i,b_i).
\end{equation}
We now split the proof in three cases.
\\ 
{\bf 1st case.} In this case, we suppose that $N_w=1$, and thus $I_{\Omega,w}=\,(a_1,b_1)$.  Consider $(\tilde{u}^{(1)}_h)_{h\in\bbN}$ be the sequence defined in \eqref{nuova2} for $i=1$, and let $u_{h}\coloneqq \bar u_h:\,(a,b)\to\R$ be defined as
\[
\bar u_h(x)\coloneqq
\begin{cases}
\tilde{u}^{(1)}_h(a_1)&\text{ if }x\in (a,a_1]\\
\tilde{u}^{(1)}_h(x)&\text{ if }x\in (a_1,b_1)\\
\tilde{u}^{(1)}_h(b_1)&\text{ if }x\in [b_1,b)\,.
\end{cases}
\]
Then  by definition, we notice that see that $(\bar{u}_h)_h\subset C^1([a,b])\subset AC([a,b])$, and  thus \eqref{approxseqW} holds true.  In the next, we will prove that \eqref{recovsequbd} is truth. Since \eqref{nuova1}, and \eqref{nuova2} hold truth, we get that the sequence $(\bar u'_h)_h$ is bounded in $L^{p}(I_{\Omega,w},w)$\,. We now show the following:
\begin{equation}\label{CCCf}
\int_{a}^{b}|\tilde u_h-u|^{p}\,(\wstar_{p})^{p-1}\,\de x\to 0\text{ as }h\to \infty\,.
\end{equation}
Indeed, since $\wstar_{p}\equiv 0$ in $\Omega\setminus I_{\Omega,w}$, we have that
\begin{align*}
\int_{a}^{b}|\tilde u_h-u|^{p}\,(\wstar_{p})^{p-1}\,\de x=
\int_{a_1}^{b_1}|\tilde u_h-u|^{p}\,(\wstar_{p})^{p-1}\,\de x.
\end{align*}
Hence, by our Poincar\'e inequality \eqref{poincareformula} with $\tilde{u}_{h}-u$ instead of $u$, and  due to $\tilde{u}_h\left(\frac{a_1+b_1}{2}\right)=\,u\left(\frac{a_1+b_1}{2}\right)$, we deduce that
\begin{equation*}
\begin{split}
\int_{a_1}^{b_1}|\tilde u_h-u|^{p}\,(\wstar_{p})^{p-1}\,\de x
&\leq \,\int_{I_{\Omega,w}}|\tilde u'_h -u' |^{p}\,w\,\de x
\\
&=\,\int_{I_{\Omega,w}}|v_h -u' |^{p}\,w\,\de x
\,.
\end{split}
\end{equation*}
 Therefore, by  applying \eqref{nuova1},  our desired conclusion \eqref{CCCf} follows. We now proceed with the second case.
 \\ 
 {\bf 2nd case.}  Let us suppose that $N_w=2$. Then $I_{\Omega,w}=\,(a_1,b_1)\cup (a_2,b_2)$. We suppose without loss of generality that $b_1\leq a_2$ and we distinguish the cases $b_1<a_2$ and $b_1=a_2$. Suppose that $(\tilde{u}^{(i)}_h)_h$ is the sequence as defined in \eqref{nuova2} for $i=1,2$. 
 
 \bigskip
 If $b_1<a_2$ , for each $h\in\bbN$, let us take $u_h=\bar u_h:\,\Omega\to\R$ be defined as
\[
\bar u_h(x)\coloneqq
\begin{cases}
\tilde{u}^{(1)}_h(a_1)&\text{ if }x\in [a,a_1)\\
\tilde{u}^{(1)}_h(x)&\text{ if }x\in [a_1,b_1)\\
\frac{\tilde{u}^{(2)}_h(a_2)-\tilde{u}^{(1)}_h(b_1)}{a_2-b_1}(x-b_1)+\tilde{u}^{(1)}_h(b_1)&\text{ if }x\in [b_1,a_2)\\
\tilde{u}^{(2)}_h(x)&\text{ if }x\in [a_2,b_2)\\
\tilde{u}^{(2)}_h(b_2)&\text{ if }x\in [b_2,b]\,.
\end{cases}
\]
It is customary to demonstrate that $(u_{h})_{h} \subset AC([a,b])$, ensuring the validity of \eqref{approxseqW}. In fact, we can derive \eqref{recovsequbd} by employing a strategy similar to the one used in the first case, while also noting that $\wstar_{p}\equiv 0$ in $\Omega\setminus I_{\Omega,w}$. 

\bigskip
In what follows, we consider the second subcase $b_1=a_2$. Let $h\in \bbN$ such that 
 $$
 \frac1h<\min\left\{\frac{b_i-a_i}4:i=1,2\right\}.
 $$
 Let $u_h=\bar u_h:\,\Omega\to\R$ be defined as
\[
 \bar u_h(x)\coloneqq
\begin{cases}
\tilde{u}^{(1)}_h(a_1)&\text{ if }x\in [a,a_1),\\
\tilde{u}^{(1)}_h(x)&\text{ if }x\in [a_1,\frac{a_1+b_1}2),\\
u(x)&\text{ if }x\in [\frac{a_1+b_1}2,b_1-\frac1h),\\
u(x)\frac{\sqrt[p]{\wstar_{p}(x)}}{\sqrt[p]{\wstar_{p}(b_1-\frac1h)}}&\text{ if }x\in [b_1-\frac1h,b_1),\\u(x)\frac{\sqrt[p]{\wstar_{p}(x)}}{\sqrt[p]{\wstar_{p}(a_2+\frac1h)}}&\text{ if }x\in [a_2,a_2+\frac1h),\\
u(x)&\text{ if }x\in [a_2+\frac1h,\frac{a_2+b_2}2),\\
\tilde{u}^{(2)}_h(x)&\text{ if }x\in [\frac{a_2+b_2}2,b_2),\\
\tilde{u}^{(2)}_h(b_2)&\text{ if }x\in [b_2,b]\,.
\end{cases}
\]
Observe that $( u_h)_h\subset AC([a,b])$, and  thus \eqref{approxseqW} can be obtained. We now  need to show \eqref{recovsequbd}. To do that, we prove that
\begin{equation}\label{CCCfhhhh}
\int_{\frac{a_1+b_1}2}^{b_1}|\tilde u_h-u|^{p}\,(\wstar_{p})^{p-1}\,\de x\to 0\text{ as }h\to +\infty\,,
\end{equation}
and 
\begin{equation}\label{CCCfhhhhvvv}
\int_{\frac{a_1+b_1}2}^{b_1}|\tilde u_h^\prime|^{p}\,w\,\de x\leq C<+ \infty\,.
\end{equation}
Indeed, we have
\begin{equation*}
\int_{\frac{a_1+b_1}2}^{b_1}|\tilde u_h-u|^{p}\,(\wstar_{p})^{p-1}\,\de x=
\int_{b_1-\frac1h}^{b_1}u^{p}\left(1-\frac{\sqrt[p]{\wstar_{p}(x)}}{\sqrt[p]{\wstar_{p}(b_1-\frac1h)}}\right)^{p}\,(\wstar_{p})^{p-1}\,\de x.
\end{equation*}
Then by Proposition \ref{propw} the weight $\wstar_{p}$ is decreasing in $[{{a_1+3b_1}\over{4}},b_1]$, and thus
\begin{equation}
\label{bound}
\frac{{\wstar_{p}(x)}}{{\wstar_{p}(b_1-\frac1h)}}\leq 1,\qquad x\in \left({b_1-\frac1h},b_1\right).
\end{equation}
 From here, we obtain that
\begin{equation*}
\int_{\frac{a_1+b_1}2}^{b_1}|\tilde u_h-u|^{p}\,(\wstar_{p})^{p-1}\,\de x\leq
\int_{b_1-\frac1h}^{b_1}u^{p}\,(\wstar_{p})^{p-1}\,\de x
\to 0\text{ as }h\to \infty\,.
\end{equation*}
Since the last conclusion can be also deduced in the interval $(a_2,\frac{a_2+b_2}2)$, then \eqref{CCCfhhhh} holds true.  We now need to prove \eqref{CCCfhhhhvvv}. Notice that
\[
 \bar u_h'(x):=
\begin{cases}
u'(x)&\text{ if }x\in [\frac{a_1+b_1}2,b_1-\frac1h)\\
\frac{1}{\sqrt[p]{\wstar_{p}(b_1-\frac1h)}}
\left(u'(x)\sqrt[p]{\wstar_{p}(x)}+\frac{u(x)(\wstar_{p})'(x) (\sqrt[p\prime]{\wstar_{p}(x)})^{-1}}{p}
\right)
&\text{ if }x\in [b_1-\frac1h,b_1).
\end{cases}
\]
Therefore

\begin{align*}
\int_{\frac{a_1+b_1}2}^{b_1}|\bar u_h^\prime|^{p}\,w\,\de x&=\int_{\frac{a_1+b_1}2}^{b_1-\frac1h}|u^\prime|^{p}\,w\,\de x+\\\nonumber
&\begin{aligned}
&+\int_{b_1-\frac1h}^{b_1}\frac{1}{\wstar_{p}(b_1-\frac1h)}
\left(u'\sqrt[p]{\wstar_{p}}+\frac{u(\wstar_{p})'(\sqrt[p\prime]{\wstar_{p}})^{-1}}{p}
\right)
^{p}\,w\,\de x
\end{aligned}\\\nonumber
&\leq\int_{\frac{a_1+b_1}2}^{b_1}|u^\prime|^{p}\,w\,\de x+
2^{p-1}\Bigg(
\int_{\frac{a_1+b_1}2}^{b_1}|u'|^{p}\frac{\wstar_{p}}{\wstar_{p}(b_1-\frac1h)}w\,\de x\\\nonumber
&+\frac{1}{p^{p}}\int_{b_1-\frac1h}^{b_1}\frac{1}{\wstar_{p}(b_1-\frac1h)}u^{p}|(\wstar_{p})'|
^{p} (\wstar_{p})^{-\frac{p}{p\prime}}\,w\,\de x\Bigg).
\end{align*}
Since \eqref{bound} holds true, we can then conclude that the second integral is finite.  Now, let us prove that the last integral is also finite.  Now, since $u^{p}(\wstar_{p})^{p-1}$ is bounded in $(b_1-1/h,b_1)$ (see Corollary 4.6 (i)) and $\wstar_{p}(b_1)=0$,  we have that

\begin{align*}
\int_{b_1-\frac1h}^{b_1}\frac{1}{\wstar_{p}(b_1-\frac1h)}u^{p}|(\wstar_{p})'|^{p}\frac{w}{(\wstar_{p})^{\frac{p}{p\prime}}}\,\de x
&=\int_{b_1-\frac1h}^{b_1}\frac{1}{\wstar_{p}(b_1-\frac1h)}u^{p}\left(\frac{(\wstar_{p})^{2}}{w^{\frac{1}{p-1}}}\right)^{p}\frac{w}{(\wstar_{p})^{\frac{p}{p\prime}}}
\,\de x\\
&=\int_{b_1-\frac1h}^{b_1}\frac{1}{\wstar_{p}(b_1-\frac1h)}u^{p}(\hat{w}_{p})^{p-1}\frac{(\hat{w}_{p})^{2}}{w^{\frac{1}{p-1}}}\,\de x\\
&\leq \frac{C_{p}}{\wstar_{p}(b_1-\frac1h)}\int_{b_1-\frac1h}^{b_1}\frac{(\hat{w}_{p})^{2}}{w^{\frac{1}{p-1}}}\,\de x.\\
\end{align*}
Now, by \eqref{derwstar} we have
\begin{align*}
(\wstar_{p})^\prime
={{\left(\wstar_{p}\right)^2}\over{w^{\frac{1}{p-1}}}} \quad\text{ a.e. in }\left(b_1-\frac1h,b_1\right),
\end{align*}
and then
\begin{align*}
\int_{b_1-\frac1h}^{b_1}\frac{1}{\wstar_{p}(b_1-\frac1h)}u^{p}|(\wstar_{p})'|^{p}\frac{w}{(\wstar_{p})^{\frac{p}{p\prime}}}\,\de x\leq & \frac{C_{p}}{\wstar_{p}(b_1-\frac1h)}\int_{b_1-\frac1h}^{b_1}|(\wstar_{p})'|\,\de x
\\
&= \frac{C_{p}}{\wstar_{p}(b_1-\frac1h)}\left[\wstar_{p}\left(b_1-\frac1h\right)-\wstar_{p}(b_1)\right]= C_{p}.
\end{align*}
Finally, we address the last case.\\
{\bf 3rd case.} In the general case where $I_{\Omega,w}=\bigcup_{i=1}^{N_w}(a_i,b_i)$ with $b_i\leq a_{i+1}$ for every $i=1,\dots,N_w-1$, it is sufficient to replicate the arguments of the 2nd case for each $i=1,\dots,N_w-1$.
\end{proof}

\section{Relaxation for finitely degenerate weights}\label{relaxp}
We consider $X=L^{p}(\Omega,(\wstar_{p})^{p-1})$ where $\wstar_{p}$ is the weight as defined in \eqref{pesop}. Let us set

\begin{equation*}
\!\!\!\!\!\! \!\!\! \!\!\!\!\!\! \! F(u)\coloneqq
\begin{cases}
\displaystyle\int_{a}^b |u'|^{p}\,w\,\de x &\text{ if } u\in {\rm AC}([a,b])\\
+\infty & \text{ if } u\in X\setminus {\rm AC}([a,b]) 
\end{cases}
\end{equation*}
and consider  the lower semicontinuous envelopes w.r.t.  $L^{p}((\wstar_{p})^{p-1})$-convergence, that is
$$
\overline{F}(u)={\rm sc^-}(L^{p}((\wstar_{p})^{p-1}))-F(u).
$$
We set
$$
D\coloneqq\{u\in L^{p}((a,b),(\wstar_{p})^{p-1}): \overline{F}(u)<+\infty
\}\,.
$$
Let us recall that, if $I_{\Omega, w}=\emptyset$, then $\wstar_{p}\equiv 0$ (see Proposition \ref{propw} (i)). This implies that $L^{p}((a,b),(\wstar_{p})^{p-1})=\{0\}$, $D=\{0\}$ and $\overline{F}(u)=0$. In the next theorem, we state an explicit formula for the relaxed functional $\overline{F}$ with respect to the chosen convergence.

\begin{teo} \label{mainfinitelocsum} 
Suppose that   $w$ is a finitely degenerate weight.
Then
\begin{equation*}
D={\rm{Dom}}_w
\end{equation*}
where ${\rm{Dom}}_w$ is defined by \eqref{domw}
and the following representation holds for the relaxed functional
\begin{equation*}
\overline{F}(u)=
\begin{cases}
\displaystyle\int_{I_{\Omega, w}} |u'|^{p}\,w\,\de x &\text{ if } u\in {\rm{Dom}}_w\\
+\infty & \text{ if } u\in L^{p}(\Omega,(\wstar_{p})^{p-1})\setminus {\rm{Dom}}_w .
\end{cases}
\end{equation*}  
\end{teo}

\begin{proof} Note that  by  Lemma \ref{lemma2} and  Proposition \ref{poincarep},  we  deduce that  $D\subseteq {\rm{Dom}}_w$. Furthermore, for every $u\in D$ one gets
$$
u\in W^{1,1}_{\text {loc}}(I_{\Omega, w})\cap L^{p}(I_{\Omega, w}, (\wstar_{p})^{p-1}),\ \ u^{p}(\wstar_{p})^{p-1}\in L^\infty(I_{\Omega, w})\,.$$ 
In the next, we show that for every $u\in L^{p}(\Omega,(\wstar_{p})^{p-1})$
$$
\displaystyle\int_{I_{\Omega, w}} |u' |^{p}\,w\,\de x\leq\overline{F}(u).
$$
By the definition of  $\overline{F}$, we directly suppose that $\overline{F}(u)<+\infty$. Therefore there exists a sequence $(u_h)\subset {\rm{Dom}}_w$ such that 
$u_h\to u$ in $L^{p}(\Omega,(\wstar_{p})^{p-1})$ and
$$
\overline{F}(u)=\lim_{h\to +\infty}F(u_h)=\lim_{h\to +\infty}\displaystyle\int_{\Omega} |u_h' |^{p}\,w\,\de x.
$$
Then, thanks to Lemma \ref{lemma2} we get up to extracting a subsequence that
$$
\displaystyle\int_{\Omega} |u' |^{p}\,w\,\de x\leq\liminf_{h\to +\infty}\displaystyle\int_{\Omega} |u_h' |^{p}\,w\,\de x=\lim_{h\to +\infty}\displaystyle\int_{\Omega} |u_h' |^{p}\,w\,\de x=\overline{F}(u)
$$
and we are done. To conclude, it remains to prove that
\begin{equation}\label{main1bb1}
\overline{F}(u)\le\,\displaystyle\int_{I_{\Omega, w}} |u' |^{p}\,w\,\de x,\quad\forall u\in {\rm{Dom}}_w\,
\end{equation}
and thus ${\rm{Dom}}_w\subseteq D.$ In fact, this is a consequence of \eqref{LipdenseW}.  Indeed, by  property (i) in Corollary \ref{proppp1}  we have that that ${\rm{Dom}}_w\subset\,L^{p}(\Omega,(\wstar_{p})^{p-1})$. Thus,
if $u\in W={\rm{Dom}}_w\cap L^{p}(\Omega,(\hat{w}_{p})^{p-1})=\,{\rm{Dom}}_w$, by \eqref{LipdenseW}, there exists a sequence $(u_h)_h\subset AC([a,b])$ such that
\eqref{convinWsplitted} holds true. Hence, from \eqref{convinWsplitted}, one has that
\begin{equation*}
\begin{split}
\overline{F}(u)&\le\,\liminf_{h\to\infty}{\overline{F}}(u_h)\le\,\lim_{h\to\infty}\int_{I_{\Omega, w}} |u'_h |^{p}\,w\,\de x= \, \int_{I_{\Omega, w}} |u' |^{p}\,w\,\de x\,,
\end{split}
\end{equation*}
and thus \eqref{main1bb1} holds true.
 \end{proof}
 
%%%%%%%%%
We consider the following functionals defined on the space $L^p(\Omega,(\wstar_{p})^{p-1})$
\begin{equation*}
\!\!\!\!\!\! \!\!\! \!\!\!\!\!\! \!\!\!   F^1(u):=
\begin{cases}
\displaystyle\int_{a}^b |u'|^p\,w\,dx &\text{ if } u\in C^1([a,b]),\\
+\infty & \text{ if } u\in L^p(\Omega,(\wstar_{p})^{p-1})\setminus C^1([a,b]) ,
\end{cases}
\end{equation*}
\begin{equation*}
\!\!\!\!\!\! \!\!\! \!\!\!\!\!\! \!  F^2(u):=
\begin{cases}
\displaystyle\int_{a}^b |u'|^p\,w\,dx &\text{ if } u\in {\rm Lip}([a,b]),\\
+\infty & \text{ if } u\in L^p(\Omega,(\wstar_{p})^{p-1})\setminus {\rm Lip}([a,b]) ,
\end{cases}
\end{equation*}
\begin{equation*}
\!\!\!\!\!\! \!\!\! \!\!\!\!\!\! \!\!  F^3(u)=
\begin{cases}
\displaystyle\int_{a}^b |u'|^p\,w\,dx &\text{ if } u\in H^1((a,b)),\\
+\infty & \text{ if } u\in L^p(\Omega,(\wstar_{p})^{p-1})\setminus H^1((a,b)),
\end{cases}
\end{equation*}  
and the corresponding lower semicontinuous envelopes w.r.t. the $L^p(\Omega,(\wstar_{p})^{p-1})$-convergence 
\begin{equation}
\label{newrelaxfunct}
\overline {F^j}(u)={\rm sc^-}(L^p(\Omega,(\wstar_{p})^{p-1}))-F_j(u)\ \ \ \ j=1,2,3\,.
\end{equation}
\begin{cor}
For every $u\in L^p(\Omega,(\wstar_{p})^{p-1})$ we have
$$
\overline{F^1}(u)=\overline{F^2}(u)=\overline{F^3}(u)=\overline{F}(u),
$$
where
$\overline{F^j}(u)$, $j=1,2,3$ are the functionals in \eqref{newrelaxfunct}.
\end{cor}
\proof 
As in the proof of Corollary 4.20 in \cite{CC} it suffices to apply the classical argument of approximation by convolution.

\qed

\textsc{Acknowledgments.}
The authors are members of  the Istituto Nazionale di Alta Matematica (INdAM), GNAMPA Gruppo Nazionale per l'Analisi Matematica, la Probabilità e le loro Applicazioni, and are partially supported by the INdAM--GNAMPA 2023 Project \textit{Problemi variazionali degeneri e singolari} and the INdAM--GNAMPA 2024 Project \textit{Pairing e div-curl lemma: estensioni a campi debolmente derivabili e diﬀerenziazione non locale}. Part of this work was undertaken while the first and third authors were visiting Sapienza University and SBAI Department in Rome. They would like to thank these institutions for the support and warm hospitality during the visits. 
\\
This study was carried out within the ''2022SLTHCE - Geometric-Analytic Methods for PDEs and Applications (GAMPA)" project – funded by European Union – Next Generation EU  within the PRIN 2022 program (D.D. 104 - 02/02/2022 Ministero dell’Università e della Ricerca). This manuscript reflects only the authors’ views and opinions and the Ministry cannot be considered responsible for them.
\\
%We also extend our appreciation to Prof. {\sc F. Serra Cassano} for his insights, valuable suggestions in the study of the problem, and his constructive comments during the preparation of this work. 

\end{document}